\documentclass{amsart}
\usepackage{mathabx,mathtools,amsrefs,mathrsfs,math,lineno}
\usepackage[a4paper,margin=41mm]{geometry}

\newcommand{\smashp}{\mathbin{\#}}
\newcommand\A{{\mathscr A}}
\newcommand\B{{\mathscr B}}
\newcommand\HH{{\mathscr H}}
\newcommand\K{{\mathscr K}}
\DeclareMathOperator{\Sym}{Sym}
\DeclareMathOperator{\gr}{gr}

\newcommand{\Pf}{\mathcal P_{\!f}}
\newcommand{\id}{\mathrm{id}}

\numberwithin{equation}{section}

\newtheorem{mainthm}{Theorem}

\newtheorem{mainprop}[mainthm]{Proposition}

\theoremstyle{definition}\newtheorem{maindefn}[mainthm]{Definition}
\newtheorem*{defn*}{Definition}

\begin{document}
\title{On extensive amenability of algebras}
\author{Laurent Bartholdi}
\date{September 22, 2026}
\address{Institut Camille Jordan, Université de Lyon 1 \textnormal{\itshape and} Section de Mathématiques, Université de Genève}
\email{laurent.bartholdi@gmail.com}

\thanks{The author gratefully acknowledges support from the SNSF Advanced Grant TMAG-2\_216487/1.}
\begin{abstract}
For a module \(M\) over a cocommutative Hopf algebra \(\A\), we define
extensive amenability by amenability of its symmetric algebra as a module
over \(\Sym(M)\smashp\A\). Using the author's coalgebraic rounding and quotient
theorems, we prove that nonzero extensively amenable modules
are amenable, that every module over an amenable Hopf algebra is
extensively amenable, and that extensive amenability is preserved and
reflected by short exact sequences. For permutation modules this
definition agrees, over every field, with extensive amenability of
group actions. A final section gives the modifications for exterior
algebras and establishes the comparison between the two definitions
in several cases.
\end{abstract}
\maketitle

Amenability of groups developed into a thriving and fundamental notion in part because of the variety of equivalent definitions of the notion: almost-invariant sets, invariant functionals, subexponential decay of return probabilities, triviality of a suitable Poisson boundary, etc., see~\cite{bartholdi:amenability}. The equivalences between the definitions lead both to useful applications and a powerful method of proof of (non)amenability for diverse examples and constructions. For example, an extension of groups is amenable if and only if both subgroup and quotient are amenable, and the implications are of varying difficulty depending on the definition used.

It turned out that it is much more convenient to study more generally the amenability of $G$-sets; most notions and results extend, and one may define amenability of a group \emph{either} as amenability of the regular $G$-set ${}_G G$ \emph{ or equivalently } as amenability of all non-empty $G$-sets. Remarkably however, amenability does not propagate through extensions: if $f\colon X\twoheadrightarrow Y$ is a map of $G$-sets, then it is not true that ${}_G X$ is amenable if and only if ${}_G Y$ and all the fibre actions ${}_{G_y}(f^{-1}(y))$ are amenable.

This led Juschenko, Matte Bon, Monod and de la Salle to introduce \emph{extensive amenability}, a stronger property of $G$-sets, which also characterizes amenable groups and does propagate through extensions~\cite{juschenko-mattebon-monod-delasalle:extensive}*{Proposition~2.4}.

To define amenability of associative algebras, we have essentially only one definition at our disposal:
\begin{defn*}[Elek~\cite{elek:amenaa}]
  Let $\A$ be an algebra over a field $\Bbbk$. The left $\A$-module $M$ is
  \emph{amenable} if for every $\epsilon>0$ and every finite
  $S\le\A$ there is a finite $F\le M$ with
  \[\dim((F+S F)/F)<\epsilon\dim F.
  \]
  (Here and throughout, ``finite'' means ``finite-dimensional'' when applied to vector spaces. The strict inclusion `$<$' prevents $F=0$.)
\end{defn*}

The case in which $\A$ is a cocommutative Hopf algebra is of particular interest, covering both group rings and universal envelopes of Lie algebras, and we restrict ourselves to this case in this article. Note that the $0$ module is never amenable. The regular module ${}_\A\A$ is amenable if and only if every nonzero $\A$-module is amenable, by Corollary~\ref{cor:all-modules} below; in this case, $\A$ itself is called \emph{amenable}. In a previous article, I prove that amenability is preserved by Hopf quotients, Hopf subalgebras, and cleft extensions, see~\cite{bartholdi:hopfamen}*{Theorems~1.9, 2.2 and~2.5}.

One soon discovers that there is little structure on amenable modules over non-amenable Hopf algebras; for example, a direct sum of an amenable and a non-amenable module is amenable, so amenability need not pass to submodules or quotients. I propose the following definition, to recover preservation under exact sequences (see Theorem~\ref{mainthm:exact}):
\begin{maindefn}
  Let $M$ be a left $\A$-module. It is \emph{extensively amenable} if the symmetric algebra $\Sym(M)$ is amenable when viewed as a $\Sym(M)\#\A$-module.
\end{maindefn}
Here $\Sym(M)$ is a Hopf algebra with each $m\in M$ primitive, and $\Sym(M)\#\A$, which from now on will be written $M\wr\A$, is $\Sym(M)\otimes\A$ qua coalgebra, with multiplication and action
\[(r\smashp a)(s\smashp b)
   =\sum r(a_{(1)}s)\smashp a_{(2)}b,
 \qquad
 (r\smashp a)\cdot s=r(as).
\]
The induced action makes $\Sym(M)$ an $\A$-module algebra and coalgebra,
and $M\wr\A$ is a cocommutative Hopf algebra.

The zero module is extensively amenable, since $\Sym(0)=\Bbbk$, although it
is not amenable. More generally,
\begin{mainprop}\label{prop:locally-finite}
Every locally finite $\A$-module (namely, a module in which every finite subspace is contained in a finite submodule) is extensively amenable.
\end{mainprop}
\begin{proof}
For given finite $S\le M\wr\A$, choose a finite $\A$-submodule $V\le M$
containing all variables in $S$'s polynomial coefficients. Set \(q=\dim V\),
and bound the degrees of those coefficients by \(d\). The $\A$-action
preserves polynomial degree. For
\(E_n=\Sym^{\le n}(V)\), the expanded space is contained in
\(\Sym^{\le n+d}(V)\), so its expansion ratio is at most
\[
  \left.\binom{n+d+q}{q}\right/\binom{n+q}{q}\to 1.\qedhere
\]
\end{proof}

The main results of this article are:
\begin{mainthm}[= Theorem~\ref{thm:ext=>amen}]
  Every nonzero extensively amenable $\A$-module is amenable.
\end{mainthm}

\begin{mainthm}[= Theorem~\ref{thm:amen=>reg}]
  Every module over an amenable cocommutative Hopf algebra is extensively amenable. In particular, its regular module is extensively amenable.
\end{mainthm}

\begin{mainthm}[= Theorem~\ref{thm:exact}]\label{mainthm:exact}
  For every exact sequence of $\A$-modules
  \[0\longrightarrow M\longrightarrow N\longrightarrow P\longrightarrow0,
  \]
  $N$ is extensively amenable if and only if $M$ and $P$ are extensively amenable.
\end{mainthm}

The following proves, in the case of extensive amenability, the corresponding result~\cite{bartholdi:aa1}*{Theorem~1.4} for usual amenability:
\begin{mainthm}[= Theorem~\ref{thm:group}]
  For every field \(\Bbbk\) and every \(G\)-set \(X\), the permutation
  \(\Bbbk[G]\)-module \(\Bbbk[X]\) is extensively amenable if and only if
  \(G\curvearrowright X\) is extensively amenable.
\end{mainthm}

In fact, there is quite a broad flexibility in the definition of extensive amenability of $G$-sets. One may replace the functor $(\Z/2)^{(X)}$ by $B^{(X)}$ for any nontrivial amenable group $B$, or by $\operatorname{Sym}_{\mathrm{fin}}(X)$, the group of finitely supported permutations of $X$, see~\cite{juschenko-mattebon-monod-delasalle:extensive}*{Theorem~1.3 and Examples~3.1--3.2}.

Other functorial constructions $\mathbf{VectorSpace}\to\mathbf{AmenableAlgebra}$ may be used in place of the symmetric algebra, provided they carry a compatible induced $\A$-action and Hopf structure. I have not gone into the general theory which is beyond my non-AI brain, but contented myself with the exterior algebra, interpreted as a Hopf superalgebra, and call $M$ \emph{$\wedge$-extensively amenable} when $\bigwedge M$ is amenable under multiplication and the induced $\A$-action. We regard $\A$ as purely even and $M$ as purely odd in the Hopf superalgebra $\bigwedge M\smashp\A$.

I have managed to prove that all the above main results remain valid for $\wedge$-extensive amenability, but \emph{not} that $M$ is extensively amenable if and only if it is $\wedge$-extensively amenable, in contrast with the group case.

\subsection*{Acknowledgments to humans and AI}
I am grateful to Mikael de la Salle for numerous stimulating discussions on this topic. My definition of extensive amenability for algebras is directly inspired from the group case.

ChatGPT 5.6 Sol was used to assist with drafting, editing, checking parts of this manuscript, and exploratory formal verification. All mathematical content was reviewed, checked, and substantially revised by the author, who is responsible for the final text.

The main results have been formalized in \textsf{Lean/Mathlib} and are available at\\
\centerline{\url{https://github.com/laurentbartholdi/hopfamenability}}

\section{Two transfer lemmas}
In equations ``$\dim((F+SF)/F)<\epsilon\dim F$'', we refer to $S$ as a \emph{testing space} and $F$ as a \emph{F\o lner space}.
By the fundamental theorem of coalgebras, finite testing spaces
in a Hopf algebra may be enlarged to finite subcoalgebras containing
\(1\). We use this reduction below.

The following is the tensor-space form of the argument proving
\cite{bartholdi:hopfamen}*{Theorem~1.9}. We record the slight strengthening needed
below.

\begin{lem}[Tensor-factor extraction]\label{lem:tensor}
Let \(\HH\) be a cocommutative Hopf algebra, let $U,V$ be left $\HH$-modules, and give \(U\otimes V\) the
diagonal \(\HH\)-action. If \(F\le \HH\) is a finite subcoalgebra
containing \(1\), then for every nonzero finite \(E\le U\otimes V\)
there is a nonzero finite \(V_0\le V\) such that
\[
 \frac{\dim FV_0}{\dim V_0}\le \frac{\dim FE}{\dim E}.
\]
Consequently, if \(U\otimes V\) contains an amenable \(\HH\)-submodule,
then \(V\) is amenable.
\end{lem}
\begin{proof}
  Put \(E^+=FE\) and choose a finite subspace \(W\le V\) containing all second tensor coefficients of \(E\) and \(E^+\).  For \(Z=E,E^+\) put \(r_Z(L)=\dim(Z\cap(U\otimes L))\). The functions \(r_Z\) are monotone and supermodular on the subspace lattice of $W$. For \(t>0\), let \(C_t(Z)\)
be the largest maximizer of \(r_Z(L)-t\dim L\), which exists by
\cite{bartholdi:hopfamen}*{Lemma~1.2}. These are also maximizers over
all finite subspaces of \(V\): replacing \(L\) by \(L\cap W\) leaves its
rank unchanged and cannot increase its dimension.

First suppose that \(F\) has a complete subcoalgebra flag
\(0=F_0<\cdots<F_s=F\). We claim that \(F C_t(E)\le C_t(E^+)\).
Write \(C=C_t(E)\), \(C^+=C_t(E^+)\), and suppose inductively that
\(F_{i-1}C\le C^+\). For \(a\in F_i\setminus F_{i-1}\), there is
a group-like \(g\in F_i\), by the codimension-one argument in
\cite{bartholdi:hopfamen}*{Proposition~1.8}, with
\(\Delta(a)-g\otimes a\in F_i\otimes F_{i-1}\). Set \(K=\{c\in C:ac\in C^+\}\), let \(q\colon V\to V/C^+\) be the quotient map, and put \(\bar a\colon C\to V/C^+\), \(\bar a(c)=q(ac)\).
Write \(\delta_a\) for the diagonal action of \(a\) on \(U\otimes V\),
and \(g_U\) for the action of \(g\) on \(U\).
The relation for \(\Delta(a)\) and the inductive inclusion
\(F_{i-1}C\le C^+\) give
\[
 (\id_U\otimes q)\delta_a(z)=(g_U\otimes\bar a)(z)
 \qquad(z\in U\otimes C).
\]
Since \(g_U\) is invertible with inverse \(S(g)_U\)
and \(\ker\bar a=K\), exactness of the tensor product gives
\[
 \ker(g_U\otimes\bar a)=U\otimes K.
\]
Moreover, \(\delta_a(E)\le E^+\), because \(a\in F\), and the same
triangular relation gives
\(\delta_a(U\otimes C)\le U\otimes(C^++aC)\).
Thus the restriction of \(\delta_a\) to \(E\cap(U\otimes C)\), followed
by quotienting by \(E^+\cap(U\otimes C^+)\), has kernel exactly
\(E\cap(U\otimes K)\). It therefore induces an injection
\[
 \frac{E\cap(U\otimes C)}{E\cap(U\otimes K)}\to
 \frac{E^+\cap(U\otimes(C^++aC))}{E^+\cap(U\otimes C^+)}.
\]
Maximality of \(C\) bounds the dimension on the left below by
\(t(\dim C-\dim K)\); rank-nullity identifies this with
\(t(\dim(C^++aC)-\dim C^+)\). Hence \(C^++aC\) is also a maximizer
for \(E^+\), and the choice of the largest maximizer gives \(aC\le C^+\).

The density/coarea identity \cite{bartholdi:hopfamen}*{Lemma~1.2} now gives
\[
 \int_0^\infty\dim C_t(E)\,dt=\dim E,\qquad
 \int_0^\infty\dim F C_t(E)\,dt\le\dim E^+.
\]
Some nonzero \(C_t(E)\) has the asserted ratio.

For the general case, split a subcoalgebra flag after a finite field extension and descend the ratio, using \cite{bartholdi:hopfamen}*{Lemmas~1.6 and~1.7}.
\end{proof}

\begin{cor}\label{cor:all-modules}
Let \(\A\) be a cocommutative Hopf algebra. The regular module
\({}_\A\A\) is amenable if and only if every nonzero left \(\A\)-module
is amenable.
\end{cor}
\begin{proof}
The reverse implication follows by taking the regular module.
Conversely, suppose that \(\A\) is amenable and let \(M\ne0\).
The linear isomorphism
\[
 T\colon\A\otimes M\to\A\otimes M,\qquad
 a\otimes m\mapsto\sum a_{(1)}\otimes a_{(2)}m,
\]
with inverse \(a\otimes m\mapsto\sum a_{(1)}\otimes S(a_{(2)})m\),
intertwines the action on the first factor with the diagonal action.
For any fixed \(0\ne m\in M\), the diagonal module therefore contains
the amenable submodule \(T(\A\otimes\Bbbk m)\simeq{}_\A\A\).
Lemma~\ref{lem:tensor} makes \(M\) amenable.
\end{proof}

For cocommutative Hopf algebras \(\K\le\B\), let \(Q=\B/\B\ker\varepsilon_\K\) be the
quotient left \(\B\)-module coalgebra, with quotient map \(q\).
The left \(Q\)-coaction on \(\B\) is \(\rho=(q\otimes\id)\Delta_\B\colon\B\to Q\otimes\B\).

\begin{lem}[Normal-basis induction]\label{lem:induction}
Let \(\K\le\B\) be cocommutative Hopf algebras. Suppose there is
a linear section \(\sigma\colon Q\to\B\) of \(q\) for which
\[
 \Theta\colon Q\otimes \K\to \B,\qquad c\otimes k\mapsto\sigma(c)k
\]
is a right \(\K\)-module and left \(Q\)-comodule isomorphism, where the
source coaction is \(\Delta_Q\otimes\id\).
If \(Q\) is amenable over \(\B\) and \(V\) is amenable over \(\K\),
then \(\B\otimes_\K V\) is amenable over \(\B\).
\end{lem}
\begin{proof}
Let \(F\le \B\) be a finite subcoalgebra containing \(1\), and fix
\(0<\delta\). By \cite{bartholdi:hopfamen}*{Theorem~1.1}, choose a
nonzero finite subcoalgebra \(C\le Q\) with
\(\dim FC<(1+\delta)\dim C\), and put \(D=FC\).
Colinearity and the module-coalgebra identity give
\[
 \rho(F\sigma(C))\le D\otimes \B.
\]
Applying \(\id_Q\otimes\Theta^{-1}\) and using $Q$-colinearity of $\Theta$, we obtain
\[(\Delta_Q\otimes\id_\K)\Theta^{-1}(F\sigma(C))\le D\otimes Q\otimes\K,\]
and applying \(\Theta\circ(\id_Q\otimes\varepsilon_Q\otimes\id_\K)\) yields
\(F\sigma(C)\le\Theta(D\otimes \K)\).
Consequently there is a finite \(T\le\K\), containing \(1\),
such that \(F\sigma(C)\le\sigma(D)T\).
Choose a finite \(E\le V\) with
\(\dim TE<(1+\delta)\dim E\).
The induced vector-space identification \(\B\otimes_\K V\simeq Q\otimes V\)
then gives, for \(W=C\otimes E\),
\[
 FW\le D\otimes TE,\qquad
 \frac{\dim FW}{\dim W}<(1+\delta)^2.
\]
Let \(\delta\) tend to zero.
\end{proof}

\section{Amenability and extensively amenable modules}

\begin{thm}\label{thm:ext=>amen}
Every nonzero extensively amenable $\A$-module is amenable.
\end{thm}
\begin{proof}
Let $M$ be a nonzero extensively amenable $\A$-module and let
\[
 d\colon \Sym(M)\to \Sym(M)\otimes M,\qquad d(m)=1\otimes m
\]
be the universal derivation. It is $\A$-linear for the diagonal action
and by the Leibniz identity satisfies
\[d(m_1\cdots m_k)=\sum_{i=1}^k m_1\cdots \widehat{m_i}\cdots m_k\otimes m_i.\]

Fix a finite \(F\le\A\) containing \(1\), a nonzero \(m\in M\),
and \(\eta>0\). Extensive amenability gives a finite
\(E\le \Sym(M)\) such that
\[
 Z=FE+F(mE)\quad\text{satisfies}\quad \dim Z<(1+\eta)\dim E.
\]
Put \(E'=E+mE\) and \(V=d(E')\). The map
\[
 E\to V\oplus V,\qquad
 u\mapsto\bigl(d(u),d(mu)\bigr)
\]
is injective: if both coordinates vanish, the Leibniz identity gives
\(u\otimes m=0\), hence \(u=0\). Thus \(\dim E\le2\dim V\).
Since \(d\) is $\A$-linear, \(d(FE')=FV\), so \(d\) induces a
surjection \(FE'/E'\twoheadrightarrow FV/V\). As \(FE'=Z\) and
\(E\le E'\), we have
\[
 \dim FV-\dim V
 \le \dim FE'-\dim E'
 <\eta \dim E\le2\eta\dim V.
\]
Hence \(\Sym(M)\otimes M\) is amenable. Lemma~\ref{lem:tensor} extracts
F{\o}lner spaces in \(M\).
\end{proof}

\begin{thm}\label{thm:amen=>reg}
Every module over an amenable cocommutative Hopf algebra is extensively
amenable. In particular, its regular module is extensively amenable.
\end{thm}
\begin{proof}
The commutative Hopf algebra \(\Sym(M)\) is amenable
\cite{bartholdi:hopfamen}*{Theorem~2.4}. The split cleft extension
\[
 \Bbbk\longrightarrow \Sym(M)\longrightarrow M\wr\A
       \longrightarrow \A\longrightarrow\Bbbk
\]
therefore makes \(M\wr\A\) amenable by \cite{bartholdi:hopfamen}*{Theorem~2.5}.
Its nonzero module \(\Sym(M)\) is amenable by
Corollary~\ref{cor:all-modules}.
\end{proof}

\section{Exact sequences}

We first recall an elementary filtered-space observation. For compatible
increasing exhaustive filtrations \((R_n)_{n\ge-1}\) and \((V_n)_{n\ge-1}\)
on an algebra \(R\) and an \(R\)-module \(V\), with \(R_{-1}=V_{-1}=0\),
amenability of \(V\) implies amenability of \(\gr V\) over \(\gr R\).
Indeed, lift a finite homogeneous subspace \(\overline F\le\gr R\)
to \(F\le R\), take a F{\o}lner space \(E\), and put \(E'=E+FE\).
Then
\[
 \gr E+\overline F\,\gr E\le\gr E',\qquad
 \dim\gr E=\dim E,\quad \dim\gr E'=\dim E'.
\]
This is the leading-symbol argument of \cite{bartholdi:hopfamen}*{Lemma~7.2};
the exhaustive nonnegative filtration ensures that the induced filtrations
on finite spaces preserve dimension. The same argument applies to
decreasing filtrations \(V=V_0\ge V_1\ge\cdots\) with
\(\bigcap_nV_n=0\): their intersections with any finite-dimensional
space eventually vanish. In particular, for an already graded algebra
and module one may take homogeneous F{\o}lner spaces, after enlarging
the finite testing space to a homogeneous one.

\begin{thm}\label{thm:exact}
For every exact sequence of $\A$-modules
\[
 0\longrightarrow M\longrightarrow N\longrightarrow P\longrightarrow0,
\]
$N$ is extensively amenable if and only if both $M$ and $P$ are extensively amenable.
\end{thm}
\begin{proof}
Suppose first that $N$ is extensively amenable. Give \(N\) the two-step filtration $0\le M\le N$, giving $M$ weight zero and \(N/M\simeq P\) in degree $1$. Then $\gr N\simeq M\oplus P$, and the induced multiplicative filtrations satisfy
\[
 \gr\Sym(N)\simeq\Sym(M\oplus P),\qquad
 \gr(N\wr\A)\simeq\Sym(M\oplus P)\smashp\A.
\]
The preceding observation proves that $M\oplus P$ is extensively amenable.
The two projections of \(M\oplus P\) induce surjective counital
module-coalgebra maps to \(\Sym(M)\) and \(\Sym(P)\). The quotient theorem
\cite{bartholdi:hopfamen}*{Theorem~1.9}, with the actions factored through \(M\wr\A\)
and \(P\wr\A\), proves both conclusions.

Conversely, take \(\B=N\wr\A\), \(\K=M\wr\A\), and
\(Q=\B/\B\ker\varepsilon_\K\simeq\Sym(P)\), with quotient map \(q\) and coaction
\(\rho=(q\otimes\id)\Delta_\B\).
A vector-space section \(P\to N\) induces an algebra and coalgebra
section \(\sigma\colon\Sym(P)\to\Sym(N)\le \B\). Multiplication gives
\[
 \Theta\colon\Sym(P)\otimes \K\xrightarrow{\ \simeq\ } \B,\qquad
 p\otimes k\mapsto\sigma(p)k.
\]
It is a right \(\K\)-module isomorphism by the polynomial normal form,
and is colinear because
\(\rho(\sigma(p))=\sum p_{(1)}\otimes\sigma(p_{(2)})\) and
\(q(bk)=q(b)\varepsilon(k)\).
The \(\B\)-action on \(\Sym(P)\) factors through \(P\wr\A\), so \({}_\B(\Sym(P))\) is amenable.
Apply Lemma~\ref{lem:induction} to the amenable \(\K\)-module \(\Sym(M)\).
The action map
\[
 \B\otimes_\K\Sym(M)\to\Sym(N),\qquad
 (r\smashp a)\otimes s\mapsto r(a\cdot s),
\]
identifies these modules, proving that $N$ is extensively amenable.
\end{proof}

\begin{cor}\label{cor:unions}
Extensive amenability passes to submodules and quotients, is closed
under extensions and direct sums, and is local on directed unions
of $\A$-submodules.
\end{cor}
\begin{proof}
The first assertions and the finite direct-sum assertion follow from
Theorem~\ref{thm:exact}. If \(M=\bigcup_iM_i\) is directed and every
\(M_i\) is extensively amenable, the coefficients of any finite subspace of \(M\wr\A\) lie in some \(\Sym(M_i)\); a F{\o}lner space there works in
\(\Sym(M)\). The converse follows by submodule permanence. Arbitrary
direct sums are directed unions of finite ones.
\end{proof}

\section{Group actions and permutation modules}\label{sec:groups}
Juschenko, Matte Bon, Monod and de la Salle call a \(G\)-set \(X\) \emph{extensively amenable} if there is a \(G\)-invariant mean on the collection \(\Pf(X)\) of finite subsets of $X$ giving mass one to
\(\{B:T\subseteq B\}\) for every finite \(T\subseteq X\), see
\cite{juschenko-mattebon-monod-delasalle:extensive}*{Definition~1.1} or \cite{bartholdi:amenability}*{Definition~9.1}. We immediately pass to an equivalent definition, see~\cite{juschenko-mattebon-monod-delasalle:extensive}*{Theorem~1.3}: with \(L=\Z^{(X)}\), $X$ is extensively amenable if and only if the affine action \(L\rtimes G\curvearrowright L\) is amenable.

In case $\A$ is a group ring $\Bbbk[G]$, our wreath product $M\wr\A=\Sym(M)\#\A$ is also known as the crossed product algebra $\Sym(M)\rtimes G$.

\begin{thm}\label{thm:group}
For every field \(\Bbbk\) and every \(G\)-set \(X\), the permutation
\(\Bbbk[G]\)-module \(\Bbbk[X]\) is extensively amenable if and only if
\(G\curvearrowright X\) is extensively amenable.
\end{thm}
\begin{proof}
  Set \(L=\Z^{(X)}\) and $R=\Bbbk[L]$, and write $u_x$ for the standard generators of $L$ and $z_x=u_x-1$. Let \(\varpi=\ker(\varepsilon\colon R\to\Bbbk)\) be the augmentation ideal, and filter \(R\) by \(\varpi^n\). The classical description of the augmentation filtration of a free abelian group gives
\[
 \bigcap_{n\ge0}\varpi^n=0,\qquad
 \gr_\varpi R=\Sym_\Bbbk(\Bbbk\otimes_\Z L)\simeq\Sym(\Bbbk[X]),\]
see~\cite{passi:gr}*{Chapter~VI and Corollary~VIII.8.2}.
The isomorphism sends \(x\) to the initial form of \(u_x-1\), so is
\(G\)-equivariant. Since \(G\) preserves \(\varpi\), the compatible
filtration \((R\smashp\Bbbk[G])_n=\varpi^n\smashp\Bbbk[G]\) gives
\begin{equation}\label{eq:gradedkG}
  \gr(R\smashp\Bbbk[G])\simeq\Sym(\Bbbk[X])\smashp\Bbbk[G].
\end{equation}

  Here \(R\smashp\Bbbk[G]\simeq\Bbbk[L\rtimes G]\), and \(R\)
is the permutation module of the affine action \(L\rtimes G\curvearrowright L\).
By the preceding criterion and the permutation-module equivalence
\cite{bartholdi:hopfamen}*{\S3}, it suffices to prove that
\(\Sym(\Bbbk[X])\) is amenable under \(\Sym(\Bbbk[X])\smashp\Bbbk[G]\)
if and only if \(R\) is amenable under \(R\smashp\Bbbk[G]\).

  Suppose first that \(\Sym(\Bbbk[X])\) is amenable under $\Bbbk[X]\wr\Bbbk[G]$. The $G$-equivariant map
  \[\Sym(\Bbbk[X])\to R,\qquad x\mapsto z_x\]
  identifies the symmetric algebra with the polynomial subalgebra $\Bbbk[z_x:x\in X]$; these are also the canonical lifts of the generators in $\gr_\varpi R\simeq\Sym(\Bbbk[X])$. If $E$ is a F\o lner space in $\Sym(\Bbbk[X])$, let $\widetilde E$ be its image in $R$. For $g\in G$ the expansion ratios of $E$ and $\widetilde E$ agree, while $\widetilde E+u_x\widetilde E=\widetilde E+z_x\widetilde E$ shows that multiplication by $u_x$ has the same boundary as multiplication by $z_x$.

For any finite-dimensional subspace \(W\) and invertible linear operators
\(a,b\) on its ambient space, put
\(b_W(a)=\dim(W+aW)-\dim W\). Then
\begin{equation}\label{eq:operator-boundary}
 b_W(ab)\le b_W(a)+b_W(b),\qquad b_W(a^{-1})=b_W(a).
\end{equation}
The first inequality follows by enlarging \(W+abW\) to \(W+aW+abW\); the additional dimension beyond $W+aW$ is at most $\dim((aW+abW)/aW)=\dim((W+bW)/W)=b_W(b)$; the second follows by applying \(a\) to \(W+a^{-1}W\).
Given a finite subset of \(L\rtimes G\), express its elements as words in
finitely many \(u_x^{\pm1}\) and elements of \(G\). Choose \(E\)
simultaneously F\o lner for the corresponding variables \(x\) and group
operators, with error small compared with the sum of the word lengths.
By \eqref{eq:operator-boundary} and
\[
 \dim\Bigl(W+\sum_j a_jW\Bigr)-\dim W\le\sum_j b_W(a_j),
\]
the image \(\widetilde E\) is F\o lner for the prescribed finite subset.
Thus \(R\) is amenable under \(R\smashp\Bbbk[G]\), so the affine action
is amenable and \(G\curvearrowright X\) is extensively amenable.

  Conversely, suppose that $R$ is amenable under $R\smashp\Bbbk[G]$. The separated-filtration argument preceding Theorem~\ref{thm:exact} proves that amenability is preserved on the associated graded, so \(\Bbbk[X]\) is extensively amenable by~\eqref{eq:gradedkG}.
\end{proof}

By the lamp-functor theorem \cite{juschenko-mattebon-monod-delasalle:extensive}*{Theorem~1.3}, or its
action-valued form \cite{bartholdi:amenability}*{Proposition~9.5}, these conditions are
also equivalent to amenability of the affine actions
\[
 L^{(X)}\rtimes G\curvearrowright L^{(X)}
\]
for \(L=\mathbb Z/2\), \(L=\mathbb Z\), or \(L=(\Bbbk,+)\).
We used \(L=\Z\) in the proof above, and will consider $L=\Z/2$ in Section~\ref{sec:exterior}.

\section{The exterior algebra version}\label{sec:exterior}
We consider a variant of the definition of extensive amenability: let us call the $\A$-module $M$ \emph{$\wedge$-extensively amenable} if the exterior algebra \(\bigwedge M\) is amenable under multiplication and the induced $\A$-action. The results above have exterior
analogues, but outside characteristic two one must retain the
supercoalgebra signs. We give the required modifications.

\subsection{Superalgebras and the coalgebra arguments}

Regard $\A$ as purely even and \(M\) as purely odd. We work in
supervector spaces, with the conventions
\[
 (u\otimes v)(u'\otimes v')
   =(-1)^{|v||u'|}uu'\otimes vv',
 \qquad
 \tau(u\otimes v)=(-1)^{|u||v|}v\otimes u.
\]
Here as usual \(|u|\) denotes degree. The exterior algebra $\bigwedge M$ is a
supercommutative, supercocommutative Hopf superalgebra with
\(\Delta(m)=m\otimes1+1\otimes m\) and \(\varepsilon(m)=0\).
Every $m\in M$ satisfies $m^2=0$ in $\bigwedge M$. Put
\[
  M\wr_\wedge\A=\bigwedge M\smashp\A,
\]
a Hopf superalgebra admitting $\bigwedge M$ as a supermodule coalgebra. Both $\bigwedge M$ and $M\wr_\wedge\A$ are $\N$-graded objects. Following the observation preceding Theorem~\ref{thm:exact}, we may restrict our considerations to degree-graded F\o lner spaces.

We unfortunately need to reprove some of the previous results in the setting of Hopf superalgebras. We gather the required conclusions in the following lemma.
\begin{lem}\label{lem:pointed}
Let \(\HH\) be a nonnegatively graded Hopf superalgebra, with parity
given by degree modulo two and cocommutative degree-zero part.
For compatibly nonnegatively graded objects and degree-preserving
maps, the conclusions of \cite{bartholdi:hopfamen}*{Theorems~1.1 and~1.9}
(coalgebraic rounding and module-coalgebra quotient permanence),
and Lemmas~\ref{lem:tensor} and~\ref{lem:induction} hold in the
superalgebra convention. In the quantitative statements the finite
testing subcoalgebras and input subspaces are assumed homogeneous,
and the resulting subcoalgebras or subspaces may be chosen homogeneous.
The amenability conclusions hold for arbitrary finite testing spaces.
\end{lem}
\begin{proof}
Every finite subspace of a graded coalgebra is contained in a finite
graded subcoalgebra: take a finite subcoalgebra containing it and then
its homogeneous hull. Indeed, if \(p_n\) denotes projection to degree
\(n\), then \(\Delta(p_nc)=\sum_{i+j=n}(p_i\otimes p_j)\Delta(c)\).
Thus arbitrary finite testing spaces may be enlarged to homogeneous
subcoalgebras containing \(1\); the filtered-space observation preceding
Theorem~\ref{thm:exact} then supplies homogeneous F\o lner spaces.

Complete homogeneous subcoalgebra flags exist by coalgebra duality: let
\(0\ne F=\bigoplus_{n\ge0}F_n\le\HH\) be a finite graded subcoalgebra,
put \(\Omega=\overline{\Bbbk}\), and consider the ordinary convolution
algebra \(R=(\Omega\otimes F)^*\), with product
\[
 (fg)(c)=\sum f(c_{(1)})g(c_{(2)}).
\]
Give \(R\) the grading \(R_n=(\Omega\otimes F_n)^*\), extending
functionals by zero on the other degrees. Then
\(R_iR_j\le R_{i+j}\), so \(R_{>0}\) is a nilpotent two-sided ideal.
Moreover, \(R/R_{>0}\simeq R_0\) is commutative, since \(F_0\) is
cocommutative. Thus, writing \(J_0\) for the radical of \(R_0\), we have
\[
 J:=\operatorname{rad}R=J_0\oplus R_{>0},
 \qquad R/J\simeq\Omega^r
\]
for some \(r\). In particular, \(J\) is homogeneous and nilpotent,
while \(R/J\) is concentrated in degree zero.

For every \(m\ge0\), the radical layer
\(J^m/J^{m+1}\) is a graded bimodule over \(R/J\). Its left and right
actions commute and preserve each degree; equivalently, it is a graded
module over
\[
 (R/J)\otimes_\Omega(R/J)^{\mathrm{op}}\simeq\Omega^{r^2}.
\]
This algebra is split semisimple and concentrated in degree zero, so
each layer is a direct sum of homogeneous one-dimensional sub-bimodules. Refining each radical layer by these
summands and taking inverse images gives a complete flag of homogeneous
two-sided ideals of \(R\). Their ordinary annihilators give, in reverse
order, a complete homogeneous subcoalgebra flag in \(\Omega\otimes F\).
Indeed, if \(I\) is a two-sided ideal and \(c\in I^\perp\), then
\(\Delta(c)\) annihilates \(I\otimes R+R\otimes I\), whence
\(\Delta(c)\in I^\perp\otimes I^\perp\).

Finally, a homogeneous basis adapted to the
flag involves only finitely many coefficients in \(\Omega\). They
generate a finite extension of \(\Bbbk\), over which the entire
homogeneous subcoalgebra flag is defined.

For a homogeneous codimension-one step \(F'\subset F'+\Bbbk a\), write
\[
 \Delta(a)-a\otimes g\in F'\otimes(F'+\Bbbk a),\qquad
 \Delta(a)-h\otimes a\in(F'+\Bbbk a)\otimes F',
\]
where \(g,h\) are even group-like elements. The super module-coalgebra
identity is
\[
 \Delta(ac)=
 \sum(-1)^{|a_{(2)}||c_{(1)}|}
       a_{(1)}c_{(1)}\otimes a_{(2)}c_{(2)}.
\]
More explicitly, let \(C\) be a finite graded subcoalgebra of an
\(\HH\)-module coalgebra, not necessarily \(\HH\)-stable, and let
\(D\le(F'+\Bbbk a)C\) be a graded subcoalgebra. Put
\(V=D+F'C\) and \(\phi:C\to(F'+\Bbbk a)C/V\), \(c\mapsto ac+V\),
as in \cite{bartholdi:hopfamen}*{Proposition~1.8}.
On underlying vector spaces, using ordinary tensor products of linear
operators, the first triangular relation gives
\((\phi\otimes g)\Delta(c)=0\) on \(\ker\phi\); the second gives
\((h\Pi^{|a|}\otimes\phi)\Delta(c)=0\), where
\(\Pi(c)=(-1)^{|c|}c\). Since \(g\) and \(h\Pi^{|a|}\) are
invertible on the ambient module, their restrictions to \(C\) are injective,
and the two tensor-kernel arguments are unchanged.
Their intersection shows that \(\ker\phi\) is a graded subcoalgebra.
The dimension inequalities in \cite{bartholdi:hopfamen}*{Proposition~1.8}
therefore telescope unchanged along the homogeneous flag.
To remove the splitting assumption, extend scalars to a finite
splitting field. The tensor-amplification and base-change arguments of
\cite{bartholdi:hopfamen}*{Lemmas~1.4--1.5} apply with a purely even
field-extension factor and graded subspaces, so preserve semistability
on the lattice of graded subcoalgebras. Apply the split transfer
inequality to scalar extensions of the original graded data and descend
it by preservation of dimensions and intersections.
The density and comparison arguments of
\cite{bartholdi:hopfamen}*{Lemmas~1.2--1.3 and the proof of Theorem~1.1}
then give rounding on the lattice of graded subcoalgebras over the
original field; no descent of the rounded subcoalgebra is needed.

Likewise, in Lemma~\ref{lem:tensor}, the leading action on
\(U\otimes V\) is \(h\Pi_U^{|a|}\otimes a\); its first factor is
invertible. Use graded subspaces in the density argument, and descend
the extracted subspace's expansion ratio by
\cite{bartholdi:hopfamen}*{Lemma~1.7}; its coefficient filtration
preserves homogeneity. This ratio descent is used only for subspaces,
not for subcoalgebras.
For an even counital coalgebra quotient \(q\colon C\to Q\), the map
\((\id\otimes q)\Delta_C\colon C\to C\otimes Q\) remains an injective
\(\HH\)-module map with the super diagonal action, proving quotient
permanence. Finally, the normal-basis proof uses only rounding,
colinearity and dimensions, and is unchanged with graded spaces.
\end{proof}

\subsection{Permanence and ordinary amenability}

\begin{thm}\label{thm:exterior}
Proposition~\ref{prop:locally-finite}, Theorems~\ref{thm:ext=>amen},
\ref{thm:amen=>reg} and~\ref{thm:exact}, and
Corollary~\ref{cor:unions} remain valid with ``extensively amenable'' replaced by ``$\wedge$-extensively amenable''.
\end{thm}
\begin{proof}
For a locally finite module, \(\bigwedge V\), with \(V\) a sufficiently
large finite invariant subspace, is already a finite invariant space
for the given test.

For an exact sequence, the weight filtration used in
Theorem~\ref{thm:exact} again has associated graded
\(\bigwedge(M\oplus P)\). (In applying Lemma~\ref{lem:pointed},
the nonnegative grading is still total exterior degree.) Quotient permanence from
Lemma~\ref{lem:pointed} gives descent to \(M\) and \(P\).
For ascent put \(\B=N\wr_\wedge\A\), \(\K=M\wr_\wedge\A\),
and \(Q=\bigwedge P\). A linear splitting \(P\to N\) induces a
graded Hopf-superalgebra section
\(\sigma\colon\bigwedge P\to\bigwedge N\le \B\).
Exterior normal form gives the right-\(\K\) isomorphism
\(\Theta(p\otimes k)=\sigma(p)k\), and
\[
 (q\otimes\id)\Delta_\B(\sigma(p))
   =\sum p_{(1)}\otimes\sigma(p_{(2)}).
\]
Thus \(\Theta\) is also a left-\(Q\) comodule map.
The super version of Lemma~\ref{lem:induction} proves ascent
exactly as before. The direct-sum and directed-union arguments
are unchanged.

If $\A$ is amenable, put \(\B=M\wr_\wedge\A\) and \(\K=\bigwedge M\).
The algebra \(\K\) is locally finite-dimensional and therefore amenable.
The quotient \(\B\otimes_\K\Bbbk=\A\), with its regular action, is amenable.
The Hopf section \(\A\to \B\) gives the right-\(\K\) normal basis
\(\A\otimes \K\to \B\); bijectivity follows from the smash-product
relations and the bijective antipode of $\A$.
The super induction lemma, applied to \(V=\K\), makes \(\B\) amenable.
Its quotient module coalgebra
\(\B\otimes_\A\Bbbk\simeq\bigwedge M\) is amenable by
Lemma~\ref{lem:pointed}.

It remains to pass from $\wedge$-extensive amenability of $M$, for \(M\ne0\), to ordinary
amenability. Put \(J=\bigwedge^{>0}M\), and let
\(\pi\colon\bigwedge M\to J\) be the $\A$-linear projection. Fix
\(0\ne m\in M\), a finite subspace \(F\le\A\) containing \(1\), and
\(\eta>0\). Choose a finite \(E\le\bigwedge M\) with
\(\dim(FE+F(mE))<(1+\eta)\dim E\), and put
\(E'=E+mE\) and \(W=\pi(E')\). As in Theorem~\ref{thm:ext=>amen},
\[
 E\to W\oplus W,\qquad
 u\mapsto\bigl(\pi(u),\pi(mu)\bigr)
\]
is injective: \(\pi(u)=0\) gives \(u\in\Bbbk1\), and then
\(\pi(mu)=mu=0\) forces \(u=0\). Thus \(\dim E\le2\dim W\).
The quotient-surjection argument of that theorem, applied to \(\pi\), gives
\[
 \dim(FW/W)\le\dim(FE'/E')<\eta\dim E\le2\eta\dim W.
\]
Letting \(\eta\) tend to zero proves amenability of \(J\).
The coproduct components of bidegree \((n-1,1)\) give an $\A$-linear
injection
\[
 J\to(\bigwedge M)\otimes M,\qquad
 m_1\wedge\cdots\wedge m_n\mapsto
 \sum_{i=1}^n(-1)^{n-i}
 (m_1\wedge\cdots\widehat m_i\cdots\wedge m_n)\otimes m_i.
\]
It is injective, since distinct wedge basis
monomials have disjoint nonempty image supports. Apply
Lemma~\ref{lem:tensor}.
\end{proof}

\subsection{Permutation modules and comparison}

\begin{thm}\label{thm:comparison-perm}
For every field \(\Bbbk\) and every \(G\)-set \(X\),
\[
 \Bbbk[X]\text{ is $\wedge$-extensively amenable}\quad\Longleftrightarrow\quad
 G\curvearrowright X\text{ is extensively amenable}.
\]
\end{thm}
\begin{proof}
Put \(L=(\mathbf Z/2)^{(X)}\), identified with \(\mathcal P_f(X)\) by support.
By~\cite{juschenko-mattebon-monod-delasalle:extensive}*{Theorem~1.3}, the action \(G\curvearrowright X\) is extensively amenable if and only if the affine action \(L\rtimes G\curvearrowright L\) is amenable.

Suppose first that this affine action is amenable. Choose an order on \(X\)
and the resulting wedge basis
\[
(e_B)_{B\in\mathcal P_f(X)}
\]
of \(\bigwedge \Bbbk[X]\). A finite testing space \(S\) in
\(\bigwedge \Bbbk[X]\# \Bbbk[G]\) is contained in the span of finitely many elements
\(r_i\# g_i\), where each \(r_i\) is an exterior monomial. Write
\(T_i=\operatorname{supp}(r_i)\in L\), and put \(h_i=(T_i,g_i)\in L\rtimes G\).

Choose a finite F\o lner set \(F\subseteq L\) for the finitely many
\(h_i\), and set
\[
E=\langle e_B:B\in F\rangle_\Bbbk.
\]
For \(B\in F\), the vector \((r_i\# g_i)e_B\) is either \(0\), or a
nonzero scalar multiple of \(e_{T_i\triangle g_iB}=e_{h_iB}\); 
indeed, it is nonzero precisely when \(T_i\cap g_iB=\varnothing\). Hence
\[
E+S E
\le
\left\langle e_B:B\in F\cup\bigcup_i h_iF\right\rangle_\Bbbk.
\]
Choosing \(F\) sufficiently F\o lner therefore makes \(E\) F\o lner for
the original testing space. Thus \(\Bbbk[X]\) is \(\bigwedge\)-extensively
amenable.

Conversely, suppose that \(\Bbbk[X]\) is \(\bigwedge\)-extensively amenable. For
\(x\in X\), define
\[
\varepsilon_x\colon\bigwedge \Bbbk[X]\to\bigwedge \Bbbk[X],
\qquad
\varepsilon_x(v)=x\wedge v,
\]
let \(p_1\colon\bigwedge \Bbbk[X]\to \Bbbk[X]\) be projection onto degree one,
and let \(x^*\colon \Bbbk[X]\to \Bbbk\) be the coordinate functional of \(x\).
Define contractions by
\[
 \iota_x=((x^*\circ p_1)\otimes\id)\Delta,
\]
and put
\[
c_x=\varepsilon_x+\iota_x.
\]
On the wedge basis, \(c_x(e_B)=\pm e_{B\triangle\{x\}}\), so \(c_x\) is a signed permutation whose underlying permutation is the toggle
\(B\mapsto B\triangle\{x\}\). Likewise \(G\) acts by signed permutations, with
underlying permutation \(B\mapsto gB\). Let
\(\Gamma\le\operatorname{GL}(\bigwedge \Bbbk[X])\) be the group generated by
all \(c_x\) and the operators induced by \(G\). Its underlying permutation
action on \(\mathcal P_f(X)\) has the same image as the affine action
\(L\rtimes G\curvearrowright L\).

Let \(T\subseteq X\) and \(S\subseteq G\) be finite. By Lemma~\ref{lem:pointed}, there is a nonzero finite graded subcoalgebra \(C\leq\bigwedge \Bbbk[X]\) which is arbitrarily F\o lner for multiplication by the \(x\in T\) and for the \(g\in S\). Since
\(C\) is a subcoalgebra, every coefficient operator of its coproduct preserves
\(C\); in particular,
\[
\iota_x C\le C.
\]
Consequently
\[
c_xC\le C+xC,
\]
so the same \(C\) is arbitrarily F\o lner for the signed permutation operators
\(c_x\), \(x\in T\), and for \(g\in S\). The boundary estimates
\eqref{eq:operator-boundary}, applied to words in these operators, show that
this monomial representation of \(\Gamma\) is amenable.

Gromov's ordering argument for linear actions
\cite{gromov:linear}*{\S3.6} applies unchanged to monomial actions:
the nonzero scalar coefficients do not affect the leading coordinates.
Therefore amenability of this monomial representation implies amenability of
its underlying permutation action on \(\mathcal P_f(X)\). Since this action
has the same image as \(L\rtimes G\curvearrowright L\), the affine action
is amenable. By \cite{juschenko-mattebon-monod-delasalle:extensive}*{Theorem~1.3},
\(G\curvearrowright X\) is extensively amenable.
\end{proof}

In characteristic two the binary-lamp correspondence is particularly
explicit:
\[
 \bigwedge\Bbbk[X]\xrightarrow{\ \simeq\ }
 \Bbbk[(\mathbb Z/2)^{(X)}],\qquad x\mapsto u_x-1.
\]
This is a \(G\)-equivariant \emph{algebra} isomorphism, not a Hopf
isomorphism for the primitive exterior coproduct:
\(\Delta(u_x-1)\) has the additional term
\((u_x-1)\otimes(u_x-1)\).

\begin{prop}\label{prop:char-two}
If $\A$ is a cocommutative Hopf algebra in characteristic $2$, then every extensively amenable $\A$-module is $\wedge$-extensively amenable. If $\A$ is generated as an algebra by its primitive elements, the converse also holds.
\end{prop}
\begin{proof}
The quotient
\(\Sym(M)\to\Sym(M)/(m^2:m\in M)=\bigwedge M\)
is an $\A$-equivariant Hopf map: in characteristic two,
cocommutativity makes the ideal generated by squares $\A$-stable.
The module-coalgebra quotient theorem
\cite{bartholdi:hopfamen}*{Theorem~1.9} gives the first implication.

For the converse put \(R=\Sym(M)\), let \(R^{[2]}\) be the subalgebra
generated by squares, and put \(\B=R\smashp\A\),
\(\K=R^{[2]}\smashp\A\).
Primitives act as derivations and annihilate every square, so $\A$
acts on \(R^{[2]}\) through its counit. The commutative Hopf algebra
\(R^{[2]}\) is amenable by \cite{bartholdi:hopfamen}*{Theorem~2.4};
since the \(\K\)-action factors through its regular action,
it is an amenable \(\K\)-module. Moreover \(\B\otimes_\K\Bbbk=\bigwedge M\),
which is amenable by hypothesis.
Choose a basis of \(M\). Sending a wedge basis monomial to its
squarefree polynomial gives a coalgebra section; separating even
exponents from squarefree exponents gives the colinear right-\(\K\)
normal basis. Lemma~\ref{lem:induction} applied to \(V=R^{[2]}\)
proves amenability of \(\B\otimes_\K R^{[2]}\simeq R\).
\end{proof}

\begin{rem}
Replacing squares by \(p\)-th powers gives the analogous equivalence
between the symmetric and truncated symmetric constructions in
characteristic \(p\), for primitively generated $\A$.
\end{rem}

I still do not know whether the symmetric and exterior definitions are equivalent.


\bib{bartholdi:aa1}{article}{
  author={Bartholdi, Laurent},
  title={On amenability of group algebras, I},
  date={2008},
  journal={Isr. J. Math.},
  volume={168},
  pages={153\ndash 165},
  review={\MR {2448055 (2010a:43001)}},
  doi={10.1007/s11856-008-1061-7},
  eprint={arXiv:math/0608302},
}

\bib{bartholdi:amenability}{article}{
  author={Bartholdi, Laurent},
  title={Amenability of groups and $G$-sets},
  conference={ title={Sequences, groups, and number theory}, },
  book={ series={Trends Math.}, publisher={Birkh\"{a}user/Springer, Cham}, },
  date={2018},
  pages={433\ndash 544},
  review={\MR {3799933}},
  eprint={arXiv:1705.04091},
}

\bib{bartholdi:hopfamen}{article}{
  author={Bartholdi, Laurent},
  title={Amenability of Lie, Group and Hopf algebras},
  date={2026},
  eprint={arXiv:2609.10373},
  status={preprint},
}

\bib{elek:amenaa}{article}{
  author={Elek, G\'abor},
  title={The amenability of affine algebras},
  journal={J. Algebra},
  volume={264},
  date={2003},
  number={2},
  pages={469\ndash 478},
  issn={0021-8693},
  review={\MR {1981416 (2004d:16043)}},
}

\bib{gromov:linear}{article}{
  author={Gromov, Misha},
  title={Entropy and isoperimetry for linear and non-linear group actions},
  journal={Groups Geom. Dyn.},
  volume={2},
  date={2008},
  number={4},
  pages={499--593},
  issn={1661-7207},
  review={\MR {2442946}},
  doi={10.4171/GGD/48},
}

\bib{juschenko-mattebon-monod-delasalle:extensive}{article}{
  author={Juschenko, Kate},
  author={{Matte Bon}, Nicol\'as},
  author={Monod, Nicolas},
  author={Salle, Mikael~{de la}},
  title={Extensive amenability and an application to interval exchanges},
  date={2015},
  eprint={arXiv:1503.04977},
}

\bib{passi:gr}{book}{
  author={Passi, Inder Bir~S.},
  title={Group rings and their augmentation ideals},
  series={Lecture Notes in Mathematics},
  publisher={Springer},
  address={Berlin},
  date={1979},
  volume={715},
  isbn={3-540-09254-4},
}



\begin{bibdiv}
\begin{biblist}
\bibselect{math}
\end{biblist}
\end{bibdiv}
\end{document}